\documentclass[10pt]{article}
\title{Null boundary controllability of a 1-dimensional heat equation with an internal point mass}
\author{Scott W Hansen
	\thanks{Department of Mathematics, Iowa State University, Ames, IA 50010, USA 		(shansen@iastate.edu) (jesusmtz@iastate.edu). 
	Funding for this research was provided in part by the National Science 
	Foundation under award number 
	DMS-1312952.}
	 ~and ~Jose de Jesus Martinez
	 \footnotemark[1]}
\date{}

\usepackage{amsmath}
\usepackage[english]{babel}
\usepackage{amsthm}
\usepackage{enumerate}
\usepackage{color,colortbl} 
\usepackage{amsfonts}
\usepackage{amssymb}
\usepackage{tikz}
\usepackage{graphicx}
\usepackage{fancyhdr}
\usepackage{empheq}
\usepackage[margin=1.0in]{geometry}

\numberwithin{thm}{subsection}
\numberwithin{cor}{subsection}
\newtheorem{lemma}{Lemma}\numberwithin{lemma}{subsection}
\newtheorem{prop}{Proposition}\numberwithin{prop}{subsection}
\newtheorem{rmk}{Remark}\numberwithin{rmk}{subsection}

\newcommand{\eq}[1]{\begin{align*}#1\end{align*}}
\newcommand{\eql}[2]{\begin{align}\label{#1}#2 \end{align}}
\def\<#1>{\langle#1\rangle}
\newcommand{\N}{\mathbb{N}}
\newcommand{\R}{\mathbb{R}}

\newcommand{\A}{\mathcal{A}}
\newcommand{\hi}{\mathcal{H}}
\newcommand{\W}{\mathcal{W}}

\newcommand{\la}{\lambda}
\newcommand{\ep}{\epsilon}

\newcommand{\ue}{{u_\ep}}
\newcommand{\ve}{{v_\ep}}
\newcommand{\ze}{{z_\ep}}

\begin{document}
%
\maketitle

\begin{abstract} We consider a linear hybrid system composed by two rods of equal length connected by a point mass. We show that the system is null controllable with Dirichlet and Neumann controls. The results are based on a careful spectral spectral analysis together with the moment method.

\vspace{2mm}
\end{abstract}

\section{Introduction}
\label{Introduction}
In this article we prove the boundary null controllability of the temperature of a linear hybrid system consisting of two wires or rods connected by a point mass.
More precisely, we consider the following system:
\eql{sys}{
&\begin{cases}
	\dot{u} - u'' = 0, & t>0, \ x\in \omega_1=(-1,0) \\
	\dot{v} - v'' = 0, & t>0, \ x\in \omega_2=(0,1) \\
	\dot{z} = v'(t,0) - u'(t,0), & t>0\\
	u(t,0) = v(t,0) = z(t), & t>0	\\
	u(t,-1) = 0,
\end{cases}
}
with either Dirichlet control
\eql{dir control}{
	v(t,1)=f(t), \qquad t>0
}
or Neumann control
\eql{neu control}{
	v'(t,1)=f(t), \qquad t>0.
}
In the above and throughout this article, $~'$ denotes spatial derivatives and $\dot{~~}$ denotes temporal derivatives. 
In addition, $u=u(t,x)$ and $v=v(t,x)$ denote the temperature on $\omega_1$ and $\omega_2$, and $z=z(t)$ denotes the temperature of the point mass.
The initial conditions at time $t=0$ are given by
\eq{
	\begin{cases}
	u^0(x) = u(0,x), & x\in \omega_1\\
	v^0(x) = v(0,x), & x\in \omega_2\\
	z^0 = z(0),
	\end{cases}
}
where the triple $\{u^0, v^0, z^0\}$ will be given in an appropriately defined function space.

System \eqref{sys} with the homogenous boundary condition
\eql{dir homogeneous}{
	v(t,1)=0, \qquad t>0
} 
can be viewed as the limit of the following ``epsilon'' system with unit density on $(-1,1)\setminus(-\ep,\ep)$ and with density $1/2\ep$ on $(-\ep,\ep)$:
\eql{approx sys}{
&\begin{cases}
	\dot{\ue} - \ue'' = 0, & t>0,\  x\in (-1,-\ep)\\
	\dot{\ve} - \ve'' = 0, & t>0,\  x\in (\ep,1)\\
	\frac{1}{2\ep}\dot{\ze} - \ze'' = 0, & t>0, \ x\in (-\ep,\ep)
\end{cases}
}
where $\ue$, $\ve$ and $\ze$ satisfy the conditions
\eql{approx conditions}{
\begin{split}
	&\ue(t,-\ep) = \ze(t,-\ep),\ \ze(t,\ep) = \ve(t,\ep),\\
	&\ue'(t,-\ep) = \ze'(t,-\ep),\ \ze'(t,\ep) = \ve'(t,\ep),\\
	&\ue(t,-1) = \ve(t,1)=0, 
\end{split}
}
for $t>0$.
In fact, in \cite{HMmodeling} the authors have shown that under appropriate assumptions of the initial data, solutions of \eqref{approx sys} with \eqref{approx conditions} converge weakly to solutions of \eqref{sys} and \eqref{dir homogeneous}.

The hybrid system \eqref{sys} is a variant of previously studied hybrid models for systems of strings and beams with interior point masses.
Hansen and Zuazua used the method of characteristics in \cite{HZwavepoint} to prove the boundary null controllability of an analogous string system with an interior point mass.
In \cite{LitTay} Littman and Taylor use transform methods to prove boundary feedback stabilization of the string mass system.
In \cite{CZbeam} and \cite{CZeulerbeam}, Castro and Zuazua used method of non-harmonic Fourier series to prove boundary controllability of systems of either Rayleigh or Euler-Bernoulli beams with interior point masses. 
We refer to \cite{LMbeam}, \cite{MoRaoCo}, \cite{CoMo}, \cite{ZhaoWeiss}, \cite{GuoIvanov} and \cite{Guo} for related results on control and stabilization of systems of beams with end masses.

Our main results are the following.
\begin{prop}\label{prop: null control}
System \eqref{sys} with either Dirichlet control \eqref{dir control} or Neumann control \eqref{neu control} is null controllable in any time $T>0$. More precisely, given $T>0$ there is a control $f\in L^2(0,T)$ such that given initial data $\{u^0,v^0,z^0\}\in L^2(\omega_1)\times L^2(\omega_2)\times \R$ we have that $\{u(T,x),v(T,x),z(T)\}=\{0,0,0\}$.
\end{prop}
The solutions in Proposition \ref{prop: null control}, are defined by transposition in the spaces $C(0,T;X_{-1/2})$ for the case of Dirichlet control and $C(0,T;\hi)$ for the case of Neumann control; see Section \ref{section: Proof of Controllability results}.

Our general approach is to reduce the control problem to a moment problem.
We consider the case of Dirichlet control and Neumann control separately in Section \ref{section: Proof of Controllability results}.

\section{Preliminaries}
\label{Preliminaries}
We begin with a discussion of well-posedness of the system \eqref{sys} with either homogeneous Dirichlet boundary condition \eqref{dir homogeneous} or Neumann boundary condition
\eql{neu homogeneous}{
	v'(t,1)=0, \qquad t>0.
}
Given $u$, $v$ and $z$ defined on $\omega_1$, $\omega_2$ and $\R$ respectively, define $y=(u,v,z,)^t$ where ${}^t$ denotes transposition. 
Let
\eq{
	\hi = L^2(\omega_1)\times L^2(\omega_2)\times\R
}
equipped with the norm
\eq{
	\|y\|^2_\hi=\|(u,v,z)\|^2_{\hi} = \|u\|^2_{\omega_1}+\|v\|^2_{\omega_2}+|z|^2
}
where $\|\cdot\|_{\omega_i}$ is the usual norm in $L^2(\omega_i)$ for $i=1,2$.
In the Dirichlet case \eqref{dir homogeneous}, let
\eql{dir spaces}{
	&\vartheta_{\omega_1}=\{u\in H^1(\omega_1)~|~u(-1)=0\}\nonumber\\
	&\vartheta_{\omega_2}=\{v\in H^1(\omega_2)~|~v(1)=0\}\\
	&\vartheta=\{(u,v)\in \vartheta_1\times\vartheta_2~|~u(0)=v(0)\}\nonumber
}
equipped with the norms
\eq{
	&\|u\|^2_{\vartheta_{\omega_i}} = \|u'\|^2_{L^2(\omega_i)}, \qquad i=1,2\\
	&\|(u,v)\|^2_\vartheta = \|u\|^2_{\vartheta_{\omega_1}}+\|v\|^2_{\vartheta_{\omega_2}}.
}
One can see that $\vartheta$ is algebraically and topologically equivalent to $H^1_0(\Omega)$ although it will be more convenient to think of $\vartheta$ as a subspace of $\vartheta_1\times\vartheta_2$. 
The space
\eq{
	\W = \{ (u,v,z)\in\vartheta\times\R~|~u(0)=v(0)=z\}
}
is a closed subspace of $\vartheta\times\R$ with norm $\|(u,v,z)\|^2_{\W} = \|(u,v)\|^2_\vartheta$.
In the Neumann case \eqref{neu homogeneous}, replace the definition of $\vartheta_{\omega_2}$ in \eqref{dir spaces} by 
\eql{neu spaces}{
	&\vartheta_{\omega_2}=H^1(\omega_2),
}
and otherwise the space $\mathcal W$ is defined the same way.    In either case, 
it is easy to show (see \cite{HMmodeling}) that the space $\W$ is densely and continuously embedded in the space $\hi$.
Define the operator $\A:D(\A)\subset\hi\rightarrow \hi$ by
\eql{dir gen}{
\A=
\begin{pmatrix} 
	d^2&0&0\\
	0&d^2&0\\
	- \delta_0d &  \delta_0d & 0\\
\end{pmatrix}
}
where $d$ denotes the (distributional) derivative operator, $\delta_0$ denotes the Dirac delta function with mass at $x=0$, and the domain $D(\A)$ of $\A$ is given in the Dirichlet case  (\ref{dir homogeneous}) by
\eql{dir domain of A}{
	D(\A)=\{y\in \W~:~ u\in H^2(\omega_1),~v\in H^2(\omega_2)\}.
}
and in the Neumann case   (\ref{neu homogeneous})   by 
\eql{neu domain of A}{
	D(\A)=\{y\in \W~:~ u\in H^2(\omega_1),~v\in H^2(\omega_2), v'(1)=0\}.
}
When $D(\A)$ is endowed with the graph-norm topology
\eq{
	\|y\|^2_{D(\A)} = \|y\|^2_{\hi}+\|\A y\|^2_{\hi}
}
it becomes a Hilbert space with continuous embedding in $\hi$. 
We can therefore write the homogeneous point-mass systems \eqref{sys}, \eqref{dir homogeneous}   and \eqref{sys}, \eqref{neu homogeneous} as
\eql{dir cauchy}{
	\dot{y}(t) = \A y(t), \quad y(0) = y^0,\quad t>0
}
where $y^0=(u^0,v^0,z^0)$.

\begin{prop}\label{prop: dir well posed}
The unbounded operator $\A$ given by \eqref{dir gen} in domain $D(\A)$ as in \eqref{dir domain of A} is a bijective, self-adjoint and dissipative operator with a compact inverse. Furthermore, $\A$ is the infinitesimal generator of a strongly continuous, compact and analytic semigroup $({\mathbb T}_t)_{t\geq 0}$. 
\end{prop}
Refer to \cite{HMmodeling} for a detailed proof of the above proposition for the Dirichlet case \eqref{sys}, \eqref{dir homogeneous} .
As a consequence of Proposition \ref{prop: dir well posed}, given initial data $y^0\in\hi$ there exists a unique solution
\eq{
	y\in C([0,\infty);\hi)
}
to the Cauchy problem \eqref{dir cauchy}.   If in addition,  $y^0\in D(\mathcal {A})$ then 
$y\in C([0,\infty), D(\mathcal {A}))$.   

In the next subsection it is shown that $\mathcal A$ has only negative eigenvalues, hence 
$-\mathcal A$ is positive, self-adjoint it provides  an isomorphism: $D({\mathcal A})\to {\mathcal H}$.  
Moreover,  fractional powers of $-\mathcal A$ are well-defined.  
Let $X_1=D(\mathcal{A})$ and for $\alpha\in [0,1]$, define $X_\alpha= D((-\mathcal{A})^\alpha)$ and $X_{-\alpha}= X_\alpha'$, the dual space relative to the pivot space $\mathcal H=X_0$ of $X_\alpha$.     
Correspondingly, the semigroup ${\mathbb T}$ remains an analytic semigroup on 
the invariant subspaces  $X_\alpha$, $0\le \alpha\le 1$, and extends continuously to an analytic semigroup on spaces $X_\alpha$, $-1\le \alpha\le 0$;  see e.g., \cite{TuWe} for full explanation.      
The norm on $X_\alpha$ is given by $\|y\|_\alpha^2 = \<(-\A)^\alpha y,(-\A)^\alpha y>_{\mathcal H}$.
In particular, $X_{1/2}$ is the completion of $X_1$ with respect to the norm 
\eq{
	\|y\|_{1/2}^2= \<-{\mathcal A} y, y>_0 .
}
Integration by parts gives
\eq{ 
	\|y\|_{1/2}^2=  \< y, y>_{\mathcal W}.
}
Thus, $X_{1/2}$ is topologically equivalent to $H^1_0(\Omega)$ in the Dirichlet case (\ref{dir homogeneous}) and $\{f\in H^1(\Omega): f(-1)=0$\} in the Neumann case (\ref{neu homogeneous}). 

\subsection{Spectral analysis for Dirichlet case \eqref{sys}, \eqref{dir homogeneous} }
\label{subsec: dir spectrum}

By Proposition \ref{prop: dir well posed}, the spectrum $\sigma(\A)$ of $\A$ is contained in the negative real axis and consists of eigenvalues $\{\la_n\}$ tending to negative infinity with corresponding eigenvectors $\{\varphi_n\}_{n\in\N}$ forming an orthogonal system for $\hi$. 
\begin{prop}\label{prop: dir eigensystem}
The eigenvalues $\{\la_n\}_{n\in\N}$ of $\A$ in the Dirichlet case   (\ref{dir homogeneous})   are distinct and given by
\eq{
	\la_{2k}=-(k\pi)^2, \quad \la_{2k-1}= -\mu_k^2 ~ \text{ for } k\in\N
}
where $\mu_k$ is the $k$-{\rm th} positive root of the characteristic equation
\eql{dir characteristic equation}{
	\mu = 2 \cot\mu.
}	
The corresponding eigenvectors are given by
\eq{
	\varphi_{2k}(x)=
		\begin{pmatrix}
			\sin(k\pi  x)\\
			\sin(k\pi x)\\
			0
		\end{pmatrix}, \quad
	\varphi_{2k-1}(x)=
		\begin{pmatrix}
			\sin( (1+x) \mu_k )\\
			\sin((1-x) \mu_k)\\
			\sin(\mu_k)\\
		\end{pmatrix}
}
and $\varphi_n\in D(\A)$ for all $n\in\N$.
\end{prop}
\begin{proof}
Look for nontrivial functions $\varphi_n=(U_n,V_n,Z_n)^t\in D(\A)$ such that $\A\varphi_n=\la_n\varphi_n$. 
We use an even index in the case that $Z_n=0$ and an odd index when $Z_n\neq0$.
The eigensystem corresponding to $Z_{2k}=0$ reduces to the problem of finding  $(U_{2k},V_{2k})$ such that
\eq{
	\begin{cases}
		U_{2k}''(x) = \la_{2k} U_{2k}(x), & x\in \omega_{1} \\
		V_{2k}''(x) = \la_{2k} V_{2k}(x), & x\in \omega_{2} \\
		U_{2k}'(0)= V_{2k}'(0)\\
		U_{2k}(0) = V_{2k}(0)=0\\
		U_{2k}(-1) = V_{2k}(1) = 0.
	\end{cases}
}
It is easy to check that $\varphi_{2k}$ satisfies the above with $\la_{2k}=-(k\pi)^2$.

Now consider the case that $Z_{2k-1}\neq0$.
The eigenvalue problem reduces to the problem of finding functions $(U_{2k-1},V_{2k-1})$, and real value $Z_{2k-1}$ such that
\eql{eigensystem}{
	\begin{cases}
		U_{2k-1}''(x) = -\mu^2_k U_{2k-1}(x), & x\in \omega_{1} \\
		V_{2k-1}''(x) = -\mu^2_k V_{2k-1}(x), & x\in \omega_{2} \\
		V'_{2k-1}(0)- U'_{2k-1}(0) =-\mu^2_k Z_{2k-1} \\
		U_{2k-1}(0) = V_{2k-1}(0) = Z_{2k-1}\\
		U_{2k-1}(-1) = V_{2k-1}(1) = 0.
	\end{cases}
}
From the boundary condition $U_{2k-1}(-1) = V_{2k-1}(1) = 0$, we have that the solution is of the form
\eq{
	&U_{2k-1}(x)=\sin\big((x+1)\mu_k\big)\\
	&V_{2k-1}(x)=C\sin\big((x-1)\mu_k\big)
}
for some constant $C$ to be determined.
The continuity condition $U_{2k-1}(0)=V_{2k-1}(0)=Z_{2k-1}$ gives
\eq{
	Z_{2k-1} =  \sin\left(\mu_k\right) = -C \sin\left(\mu_k\right).
}
Since $Z_{2k-1}$ is nonzero we have that $\mu_k$ is not a multiple of $\pi$. Furthermore, we find that $C=-1$. Then from the third equation in \eqref{eigensystem} we see that 
\eql{dir raw characteristic equation}{
	2\cot\left(\mu_k\right)=\mu_k.
}
Hence the solution to the eigensystem \eqref{eigensystem} is
\eq{
	\begin{pmatrix}
			U_{2k-1}(x)\\
			V_{2k-1}(x)\\
			Z_{2k-1}\\
		\end{pmatrix}
		=
	\begin{pmatrix}
			\sin( (1+x) \mu_k)\\
			\sin((1-x) \mu_k)\\
			\sin(\mu_k)\\
		\end{pmatrix}.
}

Finally, note that since the function $F(\mu)= 2 \cot \mu - \mu$ decreases monotonically from $+\infty $ to $-\infty$ over the interval $((k-1)\pi , k\pi)$ for all $k\in \mathbb N$,  there is exactly one root of $F$  in each interval $((k-1)\pi,k\pi)$ for all $k\in\N$.  Hence the eigenvalues 
\eq{
	\{-(k\pi)^2\}_{k\in\N}\cup\{-\mu_k^2\}_{k\in\N}
}
are distinct. 
\end{proof}

\begin{prop}\label{prop: dir eval separation}
The sequence $\{\mu_k\}$ in the Dirichlet case   (\ref{dir homogeneous})  satisfies the asymptotic estimate 
\eql{asy}{
	\mu_k=  (k-1)\pi + \frac2{k\pi} + {\mathcal O}\left(\frac1{n^2}\right).
}
Consequently, consecutive eigenvalues of $\A$ in \eqref{dir cauchy} satisfy the gap condition:
\eql{dir eval separation}{
	|\la_{n+1}-\la_n| \ge 4 + {\mathcal O}\left(\frac1{n}\right).
}
Moreover, the eigenfunctions are asymptotically normalized in the sense that 
\eq{
	\lim_{n\to\infty}\|\varphi_{n}\|=1.
}
\end{prop}
\begin{proof}
From the end of the previous proof,  $\mu_k=(k-1)\pi+\ep_k$,  where $0<\ep_k <\pi$.  The characteristic equation  \eqref{dir characteristic equation} can be rewritten as 
\eq{
	\frac{(k-1)\pi +\ep_k}2 = \cot \ep_k
}
and thus by monotonicity, 
\eq{
	(k-1)\pi/2 < \cot \ep_k < k\pi/2.
}	
Taking inverse cotangent of each term gives 
\eq{
	\arctan \frac2{k\pi} < \ep_k < \arctan \frac2{(k-1)\pi}.
}
Hence by Taylor's formula we obtain \eqref{asy}.   

The estimate \eqref{dir eval separation} can be obtained  from 
\eq{
|\lambda_{2k+1}-\lambda_{2k}| &= (\mu_{k+1} +k\pi)(\mu_{k+1}-k\pi)\\
&= \left(2k\pi +{\mathcal O}\left(\frac1{k}\right)\right)\left( \frac2{k\pi}
	+ {\mathcal O}\left(\frac1{k^2}\right)\right) \\
&= 4 +{\mathcal O}\left(\frac1{k}\right).
}
Finally, it is easy to check that $\|\varphi_{2k}\|=1$ for all $k\in\N$ and using estimate \eqref{asy} that $\|\varphi_{2k-1}\|^2 = 1 + \mathcal{O}(k^{-2})$.
\end{proof}

\subsection{Spectral analysis for Neumann case \eqref{sys}, \eqref{neu homogeneous} }
\label{subsec: neu spectrum}

As in Subsection \ref{subsec: dir spectrum}, the eigenvalues of $\A$ (denoted $\la_n$) form a discrete sequence of negative numbers tending to negative infinity with corresponding eigenvectors $\varphi_n$ which form an orthogonal system for $\hi$.

\begin{prop}\label{prop: neu eigensystem}
The eigenvalues $\{\la_n\}_{n\in\N}$ of $\A$  in the Neumann case   (\ref{neu homogeneous})    are given by $\la_n=-\mu_n^2$ where $\{\mu_n\}_{n\in\N}$ are the roots of the characteristic equation
\eql{neu characteristic equation}{
	\mu = 2\cot2\mu.
}	
The corresponding eigenvectors are given by
\eql{neu evecs}{
\begin{split}
	\varphi_{2k-1}(x)&=
		\sqrt{2}\begin{pmatrix}
			\sin(\mu_{2k-1} (x+1))\\
			\tan\mu_{2k-1} \cos(\mu_{2k-1}(x-1))\\
			\sin\mu_{2k-1}
		\end{pmatrix}\\
	\varphi_{2k}(x)&=
		\sqrt{2}\begin{pmatrix}
			\cot\mu_{2k}\sin(\mu_{2k} (x+1))\\
			 \cos(\mu_{2k}(x-1))\\
			\cos\mu_{2k}
		\end{pmatrix}
\end{split}
}
and $\varphi_n\in D(\A)$ for all $n\in\N$.
\end{prop}
\begin{proof}
The eigenvalue problem $\A\varphi_n=\la_n\varphi_n$ with $\varphi_n=(U_n,V_n,Z_n)^t\in D(\A)$ is the following system:
\eql{neu eigensystem}{
	\begin{cases}
		U_{n}''(x) = \la_{n} U_{n}(x), & x\in \omega_{1} \\
		V_{n}''(x) = \la_{n} V_{n}(x), & x\in \omega_{2} \\
		V'_{n}(0)- U'_{n}(0) =\la_{n} Z_{n} \\
		U_{n}(0) = V_{n}(0) = Z_{n}\\
		U_{n}(-1) = V'_{n}(1) = 0.
	\end{cases}
}
First note that the possibility of $Z_n=0$ leads to the trivial solution.
Hence $Z_n\neq0$ for all $n\in\N$.
Then from the first two equations and the boundary conditions we find that 
\eq{
	&U_n(x)=\sin(\mu_n(x+1))\\
	&V_n(x)=C\cos(\mu_n(x-1))
}
for some nonzero constant $C$ to be determined. 
The continuity condition $U_{n}(0)=V_{n}(0)$ gives
\eq{
	\sin\mu_n=C\cos\mu_n
}
and since $Z_n$ is nonzero for all $n\in\N$ we have that $C=\tan\mu_n$. 
Then from the third equation in \eqref{neu eigensystem} we see that 
\eq{
	\mu_n=-\tan\mu_n+\cot\mu_n
}
which is equivalent to the {\it characteristic equation} \eqref{neu characteristic equation}.
Hence the corresponding sequence of eigenvectors is 
\eq{
	\varphi_{n}(x)=
		\begin{pmatrix}
			\sin( \mu_n (1+x)  )\\
			\tan\mu_n \cos(\mu_n(x-1))\\
			\sin\mu_n\\
		\end{pmatrix}
}
which agrees with \eqref{neu evecs} after multiplying by normalizing factors $\sqrt{2}$ for $n=2k-1$ and $\sqrt{2}\cot\mu_n$ for $n=2k$.
\end{proof}
Following the ideas of Proposition \ref{prop: dir eval separation}, one can prove the following result.
\begin{prop}\label{prop: neu eval separation}
The sequence $\{\mu_k\}$   in the Neumann case (\ref{neu homogeneous})   satisfies the asymptotic estimate 
\begin{equation}\label{neu asy} 
	\mu_k=  \frac{(k-1)\pi}{2} + \frac1{k\pi} + {\mathcal O}\left(k^{-2}\right). 
\end{equation}
Consequently, consecutive eigenvalues of $\A$ in \eqref{dir cauchy} satisfy the gap condition:
\eql{neu eval separation}{
	|\la_{n+1}-\la_n| \ge \frac{n\pi^2}{2} + {\mathcal O}\left(1\right).
}
Moreover, the eigenfunctions are asymptotically normalized in the sense that 
\eq{
	\lim_{n\to\infty}\|\varphi_{n}\|=1.
}
\end{prop}

\section{Proof of Controllability results}
\label{section: Proof of Controllability results}
We begin with the case of Neumann control: \eqref{sys}, \eqref{neu control}. 

\subsection{Neumann control}
\label{neu control and obs}

The dual observation problem to \eqref{sys}, \eqref{neu control}  is
\eq{
\begin{cases}
	-\dot{\tilde u} - \tilde u'' = 0, & t>0, \ x\in \omega_1 \\
	-\dot{\tilde v} - \tilde v'' = 0, & t>0, \ x\in \omega_2 \\
	-\dot{\tilde z} = \tilde v'(t,0) - \tilde u'(t,0), & t>0\\
	\tilde u(t,0) = \tilde v(t,0) = \tilde z(t), & t>0\\	
	\tilde u(t,-1) = \tilde v'(t,1)=0, & t>0
\end{cases}
}
with terminal data at $t=T$ given by 
\eq{
	\begin{cases}
	\tilde u^T(x) = \tilde u(T,x), & x\in \omega_1\\
	\tilde v^T(x) = \tilde v(T,x), & x\in \omega_2\\
	\tilde z^T = z(T).
	\end{cases}
}
By letting $\tilde y=(\tilde u,\tilde v,\tilde z)^t$, the above problem can be written as 
\eql{neu obs cauchy}{
	-\dot{\tilde y}=\A \tilde y	, \quad \tilde y(T)=\tilde y^T\in\hi, \quad t>0.
}
Then $\tilde y\in C([0,T], \hi)$ and is given by  
\eql{neu homosol}{ 
	\tilde y(t) = {\mathbb T}(T-t) \tilde y^T; \qquad 0\leq t\leq T.
}
Let $y$ be a smooth solution of the control problem with smooth $f\in L^2(0,T)$.
 Formal integration by parts then shows 
\eq{
\begin{split}
	0&=\int^T_0\int^0_{-1} (\dot u - u'')\tilde{u}\ dx dt+\int^T_0\int^1_{0} (\dot v - v'')\tilde{v}\ dx dt\\
	&=\<y(T),\tilde{y}^T>_\hi - \<y^0,\tilde{y}(0)>_\hi
		-\int^T_0  f(t)\tilde{v}(t,1)\ dt.
\end{split}
}
Equivalently, 
\eql{ws2}{  
	\<y(T),\tilde{y}^T>_\hi =\<y^0, \mathbb T_T\tilde{y}^T>_\hi +\int^T_0  f(t)\tilde{v}(t,1)\ dt.
}
Since the functional $\ell(\tilde y):=\tilde v(1)$ is continuous on $X_{1/2}=\W$ it follows from Propositions 5.1.3 and 10.2.1 in \cite{TuWe} that for solutions of \eqref{neu obs cauchy} there exists $C>0$ for which 
\eql{neu contobs}{  
	\|\tilde v(t,1)\|_{L^2(0,T)} \le C\|\tilde y^T\|_\hi\qquad \forall~\tilde y^T \in \hi.
}
Hence, equation \eqref{ws2} uniquely defines $y(T)$ as an element of $\hi$.
Applying this definition for $s\in[0,T]$ we see 
\eql{neu solution space}{ 
	y \in C([0,T],\hi) 
} 
and there exists $C>0$ for which 
\eql{xxx}{ 
	\|y\|_{L^\infty(0,T;\hi)} \le C(\| y^0\|_\hi + \|f\|_{L^2(0,T)}).
}
As before we have the following lemma.
\begin{lemma}\label{lem: neu characterization}
The control problem \eqref{sys}, \eqref{neu control} is null controllable in time $T>0$ if and only if, for any $y^0\in \hi$ there is $f\in L^2(0,T)$ such that 
\eql{neu characterization}{
	\<y^0,\mathbb T_T\tilde{y}^T>_\hi = -\int^T_0 f(t) \tilde{v}(t,1)dt
}
holds for all $\tilde{y}^T\in \hi$, where $\tilde{y}$ is the solution to the observation problem \eqref{neu obs cauchy}.
\end{lemma}

\begin{proof}
First assume that \eqref{neu characterization} holds for all $\tilde y^T\in\hi$. 
Then by \eqref{ws2}, $y(T)=0$.
Conversely, if $f$ is a control for which $y(T)=0$, then \eqref{neu characterization} follows from equation \eqref{ws2}. 
\end{proof}

We are now ready to reduce the control problem \eqref{sys}, \eqref{neu control} to a moment problem. 
Any initial data $y^0=(u^0,v^0,z^0)^t$ in $\hi$ for the control problem can be expressed in terms of the eigenfunctions as
\eql{neu initial as fourier}{
	y^0 =\sum_{n\in\N} y^0_n \varphi_n
}
where the Fourier coefficients $\{y^0_n\}_{n\in\N}$ belong to $\ell^2$.
Let $\tilde y_n=(\tilde u_n,\tilde v_n,\tilde z_n)^t$ be the eigensolution of \eqref{neu obs cauchy} given by
\eql{neu sol to obs}{
	\tilde y_n(t,x) = e^{\la_{n}(T-t)}\varphi_n(x).
} 
In particular, note that 
\eq{
		&\tilde v_n(t,1) = 
		\begin{cases}
			\sqrt{2}e^{\la_{n}(T-t)}\tan\mu_n, & n~\text{ odd}\\
			\sqrt{2}e^{\la_{n}(T-t)}, & n~\text{ even.}\\			
		\end{cases}
}
Applying these solutions to equation \eqref{neu characterization} we obtain the following moment problem:
\eql{neu moment problem}{ 
	\frac{a_{n}}{b_n} e^{\la_{n}T}
	=\int^T_0 f(T-\tau)e^{\la_{n}\tau}d\tau, \quad n\in\N
}
where 
\eql{neu bns}{
	b_n=
	\begin{cases}
		- \tan\mu_n & n \text{ is odd}\\
		-1, & n \text{ is even}	
	\end{cases}
}
and by Proposition \ref{prop: neu eval separation}, $a_n=\|\varphi_n\|^2 y^0_n\in\ell^2$. In particular note that for $n=2k-1$
\eq{
	\tan\mu_{2k-1}=\tan\left(\frac{1}{k\pi}+\mathcal{O}(k^{-2})\right)
	=\frac{1}{k\pi}+\mathcal{O}(k^{-2})
}
and furthermore since $\tan\mu_{2k-1}\neq0$ for all $k\in\N$, there exists $\ep>0$ such that 
\eq{
	|b_n|\geq\frac{\ep}{n}, \qquad \forall n\in\N.
}
From our estimates of $\mu_n$, $\la_n$, $b_n$ and $a_n$,  it is easy to show that there are constants $K,\delta>0$ such that
\eql{neu bound moment lhs}{
	\left|\frac{a_n}{b_n}e^{\la_n T}\right|\leq K e^{-\delta n^2}, \qquad n\in\N.
}
From equations \eqref{neu asy} and \eqref{neu eval separation} we see that the series $\sum1/\la_n$ converges, and that there exists a constant $\rho>0$ such that $|\la_{k+1}-\la_k|>\rho$ for all $k\in\N$.
This implies the existence of a biorthogonal sequence $\{\theta_j(\tau)\}_{j\in\N}$ (see \cite{RussellFattBiorth}, \cite{FerManLuz}) such that 
\eql{biorthogonal sequence}{
	\int^T_0 \theta_j(\tau)e^{\la_n \tau}~d\tau=\delta_{j,n}=
	\begin{cases}
		1, & j=n\\
		0, & j\neq n.
	\end{cases}
}
By the method of Russell and Fattorini in \cite{RussellFattBiorth} we have that there are $M_1,M_2>0$ such that 
\eq{
	\|\theta_j\|\leq M_1 e^{M_2 j}.
}
It is easy to see that the above implies the convergence of 
\eq{
	f(T-\tau)=\sum_{j\in\N}\frac{a_j}{b_j}e^{\la_j T}\theta_j(\tau)
}
which provides a solution to the moment problem \eqref{dir moment problem}. 
The proof of Proposition \ref{prop: null control} for the case of Neuman control \eqref{neu control}, is a direct consequence of Lemma \ref{lem: neu characterization} and the existence of the biorthogonal sequence $\{\theta_j(\tau)\}_{j\in\N}$.

\subsection{Dirichlet Control}
\label{dir control and obs}

The dual observation problem to \eqref{sys}, \eqref{dir control}  is
\eql{dir obs sys}{
\begin{cases}
	-\dot{\tilde u} - \tilde u'' = 0, & t>0, \ x\in \omega_1 \\
	-\dot{\tilde v} - \tilde v'' = 0, & t>0, \ x\in \omega_2 \\
	-\dot{\tilde z} = \tilde v'(t,0) - \tilde u'(t,0), & t>0\\
	\tilde u(t,0) = \tilde v(t,0) = \tilde z(t), & t>0\\	
	\tilde u(t,-1) = \tilde v(t,1)=0, & t>0
\end{cases}
}
with terminal data at $t=T$ given by 
\eql{dir obs terminal data}{
	\begin{cases}
	\tilde u^T(x) = \tilde u(T,x), & x\in \omega_1\\
	\tilde v^T(x) = \tilde v(T,x), & x\in \omega_2\\
	\tilde z^T = z(T),
	\end{cases}
}
and observation $Y(t)= \tilde v'(t,1)$.  
By letting $\tilde y =(\tilde u, \tilde v, \tilde z)^t$, the above problem can be written as a Cauchy problem as
\eql{dir obs cauchy}{
	-\dot{\tilde y}=\A {\tilde y}	, \quad \tilde y(T)=\tilde y^T , \quad t>0.
}
If $\tilde y^T\in X_{1/2}={\mathcal W}$ then $\tilde y\in C([0,T], X_{1/2})$ is given by  
\eql{homosol}{ 
	\tilde y(t) = {\mathbb T}(T-t) \tilde y^T; \qquad 0\leq t\leq T.
}
Let $y$ be a smooth solution of the control problem with smooth $f\in L^2(0,T)$ and let $\tilde y$ be solution of the dual problem \eqref{dir obs cauchy}.   
Integration by parts as earlier results in the identity
\eql{ws1}{  
	\<y(T),\tilde{y}^T> =\<y^0, \mathbb T_T\tilde{y}^T>_\hi -\int^T_0  f(t)\tilde{v}'(t,1)\ dt
}
where $ \< \cdot, \cdot> $ denotes the duality pairing in $X_{-1/2}\times X_{1/2}$.

In the case of the heat equation 
\eq{
\begin{cases}
	\dot q =q'' & 0<x<1,\ t>0\\
	q(t,0)= q(t,1)= 0 &t>0\\
	q(0,x)= q^0\in H^1_0(0,1) & 0<x<1
\end{cases}
}
it is well known (e.g. \cite{Evans}) that for each $T>0$ there exists $C>0$ for which
\eq{
	\|q'(\cdot,1)\|_{L^2(0,T)}\le C \|q^0\|_{H^1_0(0,1)}.
}
One can verify that the same estimate  holds for solutions of \eqref{dir obs cauchy} in the sense that
there exists $C>0$ for which 
\eql{contobs}{  
	\|\tilde v'(t,1)\|_{L^2(0,T)} \le C\|\tilde y^T\|_{1/2}\qquad \forall~\tilde y^T \in X_{1/2}.
}
Since the semigroup $\mathbb T$ is strongly continuous on $X_{1/2}=H^1_0(\Omega)$ it follows that  the identity \eqref{ws1} defines the value $y(s)$ for all $s\in [0,T]$ as an element of $X_{-1/2}$  for which there exists $C>0$ such that 
\eq{ 
	\|y\|_{L^\infty(0,T;X_{-1/2})} \le C(\| y^0\|_\hi + \|f\|_{L^2(0,T)})
}
and moreover
\eql{dir trans sol space}{
	y \in C([0,T],X_{-1/2}). 
} 
The above estimate \eqref{dir trans sol space} is sometimes referred to as admissibility of the boundary control operator corresponding to Dirichlet control, and can also be derived in the framework of ``well posed boundary control systems"; see of \cite[    Prop. 10.7.1  ]{  TuWe}.

Analogous to Lemma \ref{lem: neu characterization} the following lemma characterizes  the problem of null controllability of \eqref{sys}, \eqref{dir control} in terms of the solution $\tilde y$ of the observation problem \eqref{dir obs cauchy}.
\begin{lemma} 
\label{lem: dir characterization}
The control problem \eqref{sys}, \eqref{dir control} is null controllable in time $T>0$ if and only if, for any $y^0\in \hi$ there is $f\in L^2(0,T)$ such that 
\eql{dir characterization}{
	\<y^0,\mathbb T_T\tilde{y}^T>_\hi = \int^T_0 f(t) \tilde{v}'(t,1)dt
}
holds for all $\tilde{y}^T\in \hi$, where $\tilde y =(\tilde u, \tilde v, \tilde z)^t$ is a solution of \eqref{dir obs cauchy}.
\end{lemma}

We are now ready to reduce the control problem \eqref{sys}, \eqref{dir control} to a moment problem. 
Any initial data $y^0=(u^0,v^0,z^0)^t$ in $\hi$ for the control problem can be expressed in terms of the eigenfunctions as
\eql{dir initial as fourier}{
	y^0 =\sum_{n\in\N} y^0_n \varphi_n
}
where the Fourier coefficients $\{y^0_n\}_{n\in\N}$ belong to $\ell^2$.
Let $\tilde y_n=(\tilde u_n,\tilde v_n,\tilde z_n)^t$ be the eigensolution of \eqref{dir obs cauchy} given by
\eql{dir sol to obs}{
	\tilde y_n(t,x) = e^{\la_{n}(T-t)}\varphi_n(x).
} 
In particular, note that
\eq{
	\tilde v'_n(t,1)=
	\begin{cases}
		e^{\la_{2k}(T-t)}k\pi (-1)^{k}, & n=2k\\
		-e^{\la_{2k-1}(T-t)}\mu_{k}, & n=2k-1.
	\end{cases}
}
We plug these solutions into equation \eqref{dir characterization} to obtain the corresponding moment problem
\eql{dir moment problem}{ 
	a_{n} e^{\la_{n}T}
	=b_n \int^T_0 f(T-\tau)e^{\la_{n}\tau}d\tau 
}
for all $n\in\N$ where 
\eql{bns}{
	b_n=
	\tilde v'_n(T,1)=
	\begin{cases}
		(-1)^k k\pi, & n=2k\\
		-\mu_k, & n =2k-1	
	\end{cases}
}
and by Proposition \ref{prop: dir eval separation}, $a_n=\|\varphi_n\|^2 y^0_n\in\ell^2$.
Again, it is easy to show that there exists constants $K,\delta>0$ such that \eqref{neu bound moment lhs} holds.
From equations \eqref{asy} and \eqref{dir eval separation} we see that the series $\sum1/\la_n$ converges, and that there exists a constant $\rho>0$ such that $|\la_{k+1}-\la_k|>\rho$ for all $k\in\N$.
This implies the existence of a biorthogonal sequence $\{\theta_j(\tau)\}_{j\in\N}$ such that there are constants $M_1,M_2>0$ such that 
\eq{
	\|\theta_j\|\leq M_1 e^{M_2 j}.
}
Hence, as earlier, 
\eq{
	f(T-\tau)=\sum_{j\in\N}\frac{a_j}{b_j}e^{\la_j T}\theta_j(\tau)
}
converges and provides a solution to the moment problem \eqref{dir moment problem}. 
The proof of Proposition \ref{prop: null control} for the case of Dirichlet control \eqref{dir control}, is a direct consequence of Lemma \ref{lem: dir characterization}.


\begin{rmk}
The numbers $b_k$ in \eqref{bns} and \eqref{neu bns}, are called {\it control input coefficients} and can be viewed as Fourier coefficients of an element $b$ of $X_{-1}$ for which the control problem \eqref{sys} with either \eqref{dir control} or \eqref{neu control} can be formulated as 
\eq{
	\dot y=\A y + bf, \quad y(0)=y^0.
}
In the Neumann case, the input element is {\it admissible} on the state space $X_0=\hi$, or equivalently that \eqref{neu solution space} and \eqref{xxx} hold.
In the Dirichlet case, $b$ is admissible on the state space $X_{-1/2}$.
Both of these spaces are slightly suboptimal in the sense that the Carleson measure criterion due to Ho and Russell \cite{HoRussell} and Weiss \cite{WeissAdm} can be used as in \cite{HansenZhang} to show admissibility holds in the spaces $X_{1/4}$ and $X_{-1/4}$ respectively for the Neumann and Dirichlet control problems.\end{rmk}





\end{document}